\documentclass[12pt]{article}

\usepackage{amsmath}
\usepackage{amsthm}
\usepackage{amsfonts}
\usepackage{amssymb}
\usepackage{enumerate}
\usepackage{setspace}
\usepackage{comment}
\usepackage[percent]{overpic}
\usepackage{graphicx}
\usepackage{float}
\usepackage{tikz}
\usepackage{calc}
\usepackage{pgfplots}
\usepackage{subfigure}
\usepackage{changepage}
\tikzstyle{vertex}=[circle, draw, inner sep=0pt, minimum size=6pt]

\tikzstyle{fillvertex}=[circle, draw, inner sep=0pt, minimum size=6pt, fill=black]

\tikzstyle{bigvertex}=[circle, draw, inner sep=0pt, minimum size=10pt]

\usepackage{array}
\newcolumntype{L}[1]{>{\raggedright\let\newline\\\arraybackslash\hspace{0pt}}m{#1}}
\newcolumntype{C}[1]{>{\centering\let\newline\\\arraybackslash\hspace{0pt}}m{#1}}
\newcolumntype{R}[1]{>{\raggedleft\let\newline\\\arraybackslash\hspace{0pt}}m{#1}}

\newtheorem{theorem}{Theorem}[section]

\newtheorem{lemma}[theorem]{Lemma}

\newtheorem{observation}[theorem]{Observation}

\theoremstyle{definition}
\newtheorem{definition}{Definition}[section]

\theoremstyle{remark}

\newcommand{\nrel}{\mathrm{NRel}}
\newcommand{\con}{\mathrm{C}}
\renewcommand{\Re}{\mathrm{Re}}
\renewcommand{\Im}{\mathrm{Im}}

\newcommand{\lexi}{\cdot}

\begin{document}
\title{On the roots of the node reliability polynomial}
\author{Jason Brown and Lucas Mol \\ Dalhousie University}
\date{}

\maketitle

\begin{abstract}
Given a graph $G$ whose edges are perfectly reliable and whose nodes each operate independently with probability $p\in[0,1],$ the \textit{node reliability} of $G$ is the probability that at least one node is operational and that the operational nodes can all communicate in the subgraph that they induce; it is the analogous node measure of robustness to the well studied \textit{all-terminal reliability}, where the nodes are perfectly reliable but the edges fail randomly. In sharp contrast to what is known about the roots of the all-terminal reliability polynomial, we show that the node reliability polynomial of any connected graph on at least three nodes has a nonreal polynomial root, the collection of real roots of all node reliability polynomials is unbounded, and the collection of complex roots of all node reliability polynomials is dense in the entire complex plane.\\


\end{abstract}

\section{Introduction}

The most common measure of robustness of a network to random failures of components is {\em all-terminal reliability}.  For a (finite, undirected) graph $G$ where nodes are always operational and edges are independently operational with probability $p\in[0,1]$, the all-terminal reliability of $G$ is the probability that all of the nodes can communicate with one another (the probability, i.e., that at least a spanning tree is operational).  The all-terminal reliability of a graph is always a polynomial in $p$, and algorithms for calculating and efficiently estimating the function have been a major focus in the area (see, for example, \cite{ColbournBook} and \cite{BrownColbournSurvey}). 
Recent attention has been drawn to the roots of all-terminal reliability polynomials, where it is known that
\begin{itemize}
\item every graph has a subdivision whose all-terminal reliability has all real roots \cite{BrownColbournLogConcave}; and
\item the real roots of all-terminal reliability are bounded.  In fact, it was proven in \cite{BCConjecture} that they are contained in the set $\{0\} \cup (1,2].$
\end{itemize}
Moreover, it was conjectured in \cite{BCConjecture} that the roots of all-terminal reliability always lie in the disk $|z-1| \leq 1$, and while the conjecture is known to be false \cite{BCFalse}, the roots found outside of the disk are only outside of the disk by a slim margin.  It still seems likely that the roots of all-terminal reliability polynomials are bounded.

There is an analogous notion for robustness in a network with node rather than edge failures. Given a graph $G$ in which edges are always operational and nodes are independently operational with probability $p$, the \textit{node reliability} of $G,$ denoted $\nrel(G;p),$ is the probability that at least one node is operational and that all of the operational nodes can communicate with one another (the probability, i.e., that  the set of operational nodes is nonempty and induces a connected subgraph of $G$).  Node reliability has sometimes been called \textit{residual node connectedness reliability} in the literature.  As with the all-terminal reliability, the node reliability of a graph is always a polynomial in $p$. For example, the node reliability of the complete bipartite graph $K_{n,n}$ is given by 
\[ \nrel(K_{n,n};p) = 2np(1-p)^{2n-1} + (1-(1-p)^{n})^{2},\]
as either a single node is operational, or at least one node in each cell is operational. 

Work on node reliability has centered on complexity issues, polynomial time algorithms for restricted families of graphs, and on finding uniformly best graphs.  It has been shown that the problem of computing the sequence of coefficients of the node reliability polynomial is \#P-complete, even for the class of planar and bipartite graphs \cite{NodeRelComplexity2}.  On the other hand, polynomial time algorithms for computing the node reliabilities of certain restricted families of graphs, including trees and series-parallel graphs, have been found \cite{NodeRelComplexity1}.  Work on the existence and identification of uniformly best graphs in particular classes has been carried out in \cite{NodeRel1,NodeRel2,NodeRel3}.

In this paper, we show that while node reliability behaves quite similarly to all-terminal reliability in some ways, the roots of the node reliability polynomial could not be more different from the roots of the all-terminal reliability polynomial.  We show that
\begin{itemize}
\item for every $n \geq 3$, the node reliability polynomial of {\it any} connected graph of order $n$ (that is, on $n$ nodes) has a nonreal root;
\item the real roots of node reliability polynomials are unbounded; and 
\item the closure of the collection of roots of node reliability polynomials is the entire complex plane.
\end{itemize}

\section{Node reliability and connected set polynomials}

We will require some elementary background on node reliability. Observe that
\begin{align*} 
\nrel(G;p) &= \sum_{C \in \mathcal{C}}p^{|C|}(1-p)^{n-|C|},\label{noderelexpansion}
\end{align*}
where $\mathcal{C}$ is the collection of all \textit{connected sets} of $G$ (a \textit{connected set} of $G$ is a nonempty subset of nodes of $G$ that induces a connected subgraph of $G$).  We can also write
\begin{eqnarray}
\nrel(G;p) & = & \sum_{k=1}^nc_kp^k(1-p)^{n-k},\label{Cformnoderel}
\end{eqnarray}
where $c_k$ is the number of connected sets of order $k$ in $G$ for $k\in\{1,\hdots,n\}.$  The following straightforward observation, giving the exact values of certain coefficients of this form of node reliability, was made in \cite{Stivaros}.

\begin{observation}\label{ConnectedBasic}
Let $G$ be a connected graph of order $n$ and size $m$, let $\tau$ be the number of triangles of $G,$ and let $t$ be the number of cut nodes of $G.$ Then
\begin{enumerate}[\normalfont(i)]
\item $c_1=n,$
\item $c_2=m,$
\item $c_3=\left(\displaystyle\sum_{v\in V(G)}\tbinom{\mathrm{deg}(v)}{2}\right)-2\tau,$
\item $c_{n-1}=n-t,$ and
\item $c_n=1.$\hfill \qed
\end{enumerate}
\end{observation}

We call the roots of the node reliability polynomial of a graph $G$ the \textit{node reliability roots} of $G.$  In order to obtain results on node reliability roots we often use a related polynomial.  Form (\ref{Cformnoderel}) of the node reliability leads us to introduce the related generating polynomial for the collection of connected sets,
\[
\con(G;x)=\sum_{k=1}^n c_kx^k,
\]
which we call the \textit{connected set polynomial}.  Given the node reliability polynomial of a graph $G$, the connected set polynomial of $G$ is easy to obtain, and vice versa, as\begin{align*}
\nrel(G;p)=(1-p)^n\cdot\con\left(G;\tfrac{p}{1-p}\right)
\end{align*}
and
\begin{align*}
\con(G;x)=(1+x)^n\cdot\nrel\left(G;\tfrac{x}{1+x}\right).
\end{align*}
The connected set polynomial is in some ways easier to analyze than the node reliability polynomial, and results on the roots of the connected set polynomial will have immediate implications for node reliability roots.  We call the roots of the connected set polynomial of a graph $G$ the \textit{connected set roots} of $G.$

\section{Real roots of the node reliability polynomial}

Many graphs have nonzero real all-terminal reliability roots; in fact, it was shown in \cite{BrownColbournLogConcave} that every connected graph has a subdivision of edges such that the resulting graph's all-terminal reliability polynomial has all real roots.  We show in contrast that the node reliability polynomial of {\em any} connected graph on at least three nodes has a nonreal root. In order to do so, we will first prove a similar result for the connected set polynomial of any such graph.  The straightforward transformation from the connected set polynomial to the node reliability polynomial will yield the desired result.  

\begin{lemma}\label{NonrealConnectedSetRoot}
Let $G$ be a connected graph of order $n\geq 3.$  Then $\con(G;x)$ has a nonreal root.
\end{lemma}

\begin{proof}
Let $\con(G;x)=\displaystyle\sum_{k=1}^n c_kx^k=x\displaystyle\sum_{k=1}^n c_kx^{k-1}.$  Suppose, to reach a contradiction, that $\con(G;x)$ has all real roots.  Since $c_k>0$ for all $k\in\{1,\hdots,n\},$ all roots of $\con(G;x)/x=\displaystyle\sum_{k=1}^n c_kx^{k-1}$ must be strictly negative.  We will need the following straightforward result (see, for example, \cite[p. 265]{Mieghem}): 

\begin{adjustwidth}{0.2in}{0.2in}
Let $f(z)=\displaystyle\sum_{k=0}^n  a_kz^k$ with $a_n\neq 0$.  If all zeros of $f$ are real and positive then 
\[ \frac{a_{n-1}}{a_n}\cdot\frac{a_1}{a_0}\geq n^2.\]
\end{adjustwidth}

\noindent An immediate corollary (by considering $f(-x)$) is that if all zeros of $f$ are real and negative then
\[
\frac{a_{n-1}}{a_n}\cdot\frac{a_1}{a_0}\geq n^2.
\]

\noindent It follows that if $\con(G;x)$ has all real roots, we must have 
\[
\frac{c_{n-1}}{c_n}\cdot\frac{c_2}{c_1}\geq (n-1)^2.
\]
However, from Observation \ref{ConnectedBasic},
\[
\frac{c_{n-1}}{c_n}\cdot\frac{c_2}{c_1}=(n-t)\frac{m}{n},
\]
where $m$ is the number of edges and $t$ is the number of cut nodes of $G$.  Since $m\leq\binom{n}{2}$ and $n\geq 3,$
\[
\frac{(n-t)m}{n}\leq \frac{(n-t)(n-1)}{2}\leq \frac{n(n-1)}{2}<(n-1)^2,
\]
a contradiction.  Therefore, $\con(G;x)/x$ (and hence $\con(G;x)$ itself) must have a nonreal root.
\end{proof}

\begin{theorem}\label{NonrealNodeRoot}
Let $G$ be a connected graph of order $n.$  If $n\geq 3$ then $\nrel(G;p)$ has a nonreal root.
\end{theorem}

\begin{proof}
By Lemma \ref{NonrealConnectedSetRoot}, the connected set polynomial $\con(G;x)$ has some nonreal root $\zeta.$  Let $p_\zeta=\frac{\zeta}{1+\zeta},$ which is well defined as $\zeta\neq -1.$  The reader can verify that $p_\zeta$ is real if and only if $\zeta$ is real, so $p_\zeta$ is nonreal.  We evaluate
\begin{align*}
\nrel(G;p_\zeta)&=(1-p_\zeta)^n \con(G;\tfrac{p_{\zeta}}{1-p_\zeta})\\
&=(1-p_\zeta)^n\con\left(G;\frac{\tfrac{\zeta}{1+\zeta}}{1-\tfrac{\zeta}{1+\zeta}}\right)\\
&=(1-p_\zeta)^n\con(G;\zeta)\\
&=0,
\end{align*}
which shows that $p_\zeta$ is a nonreal root of $\nrel(G;p)$.
\end{proof}

Since nonreal roots of polynomials with real coefficients come in conjugate pairs, every connected graph on $3$ or more nodes has at least \textit{two} nonreal node reliability roots.  
Theorem \ref{NonrealNodeRoot} provides the first major difference between the roots of node reliability polynomials and all-terminal reliability polynomials.  While every connected graph has a subdivision with all real all-terminal reliability roots, the only connected graphs with all real node reliability roots are $K_1$ and $K_2.$ We remark that there are, in fact, graphs that have {\it all} of their nonzero node reliability roots being nonreal; for example, the graph $K_{2n+1}$ satisfies this property for any $n\geq 1.$

We turn now to the location of the real roots of node reliability. It was shown in \cite{BCConjecture} that the real roots of all-terminal reliability polynomials for connected (multi)graphs are contained in $\{0\}\cup(1,2]$. It is obvious that every graph has a real node reliability root, namely $0.$ Moreover, it is easy to see that no graph on $n\geq 1$ nodes has a node reliability root in the open interval $(0,1)$, as for any $p\in(0,1),$ 
\[
\nrel(G;p)=\sum_{k=1}^n c_kp^k(1-p)^{n-k} > 0
\]
since $c_1=n,$ and all other coefficients are nonnegative.

However, in sharp contrast to what is known for all-terminal reliability, we demonstrate that the node reliability polynomial can have arbitrarily large real roots.  

\begin{theorem}\label{LargeNodeRelRoot}
The collection of real node reliability roots is unbounded.
\end{theorem}

\begin{proof}
We prove that for all $n \geq 2$, the cycle $C_{2n+1}$ has a node reliability root in the interval $(2n^2-1,2n^2).$  We begin by finding a convenient closed form for the node reliability of $C_{2n+1}$.  Note first of all that the cycle $C_N$ satisfies $c_k=N$ for all $k\in\{1,\hdots,N-1\},$ as the connected sets of order $k$ are exactly the $N$ sets of $k$ consecutive nodes.  And of course $c_N=1$ for any connected graph on $N$ nodes.  Thus we have
\begin{align*}
\nrel(C_{2n+1};p)&=p^{2n+1}+(2n+1)\sum_{k=1}^{2n}p^k(1-p)^{2n+1-k}\\
&=p^{2n+1}+(2n+1)(1-p)^{2n+1}\sum_{k=1}^{2n}\left(\tfrac{p}{1-p}\right)^k.
\end{align*}
Using the basic sum identity
\[
\sum_{k=1}^nx^k=\frac{x^{n+1}-x}{x-1},
\]
we obtain
\begin{align*}
\nrel(C_{2n+1};p)&=p^{2n+1}+(2n+1)(1-p)^{2n+1}\frac{\left(\tfrac{p}{1-p}\right)^{2n+1}-\tfrac{p}{1-p}}{\tfrac{p}{1-p}-1}\\
&=p^{2n+1}+(2n+1)\left[p^{2n+1}-p(1-p)^{2n}\right]\cdot\tfrac{1-p}{2p-1}\\
&=\tfrac{p^{2n+1}}{2p-1}\left[2n-2np+p+(2n+1)(p-1)\left(1-\tfrac{1}{p}\right)^{2n}\right].
\end{align*}

Substituting $p=2n^2-1$ gives 
\begin{align}\label{c2n2minus1}
\nrel(C_{2n+1};2n^2-1)
  & = \frac{(2n^2-1)^{2n+1}}{4n^2-3}\left[f(n)+g(n)\left(1-\tfrac{1}{2n^2-1}\right)^{2n}\right],
\end{align}
where
\[
f(n)=-4n^3+2n^2+4n-1 \ \mbox{ and } \ g(n)=4n^3+2n^2-4n-2.
\]
Using the inequality
\[
\left(1-\tfrac{1}{y}\right)^N<1-\tfrac{N}{y}+\tfrac{\binom{N}{2}}{y^2}
\]
for $y > N$, we obtain 
\[
\left(1-\tfrac{1}{2n^2-1}\right)^{2n}<1-\tfrac{2n}{2n^2-1}+\tfrac{n(2n-1)}{(2n^2-1)^2}.
\]
Applying this inequality to the right side of (\ref{c2n2minus1}), and then expanding and simplifying  yields
\begin{align*}
\nrel(C_{2n+1};2n^2-1)
  & < \tfrac{(2n^2-1)^{2n+1}}{4n^2-3}\left[f(n)+g(n)\left(1-\tfrac{2n}{2n^2-1}+\tfrac{n(2n-1)}{(2n^2-1)^2}\right)\right]\\
&=\tfrac{(2n^2-1)^{2n-1}}{4n^2-3}\left[-4n^4+2n^3+8n^2-2n-3\right].
\end{align*}
It is easily verified using a computer algebra system that the quartic 
\[
h(n)=-4n^4+2n^3+8n^2-2n-3
\]
is negative for $n\geq 2$, and we conclude that $\nrel(C_{2n+1};2n^2-1)$ is negative for $n\geq 2.$

On the other hand, substituting $p=2n^2$ gives
\begin{align}\label{c2n2}
\nrel(C_{2n+1};2n^2) & = \frac{(2n^2)^{2n+1}}{2n-1}\left[(-2n^2+2n)+(2n^2-1)\left(1-\tfrac{1}{2n^2}\right)^{2n}\right].
\end{align}
Using the inequality
\[
\left(1-\tfrac{1}{y}\right)^N>1-\tfrac{N}{y}+\tfrac{\binom{N}{2}}{y^2}-\tfrac{\binom{N}{3}}{y^3}
\]
for $y > N$, we obtain 
\[
\left(1-\tfrac{1}{2n^2}\right)^{2n}>1-\tfrac{1}{n}+\tfrac{2n-1}{4n^3}-\tfrac{(2n-1)(n-1)}{12n^5}.
\]
Applying this inequality to the right side of (\ref{c2n2}) and simplifying, we find
\begin{align}\label{c2n2ineq}
\nrel(C_{2n+1};2n^2) > \frac{(2n^2)^{2n+1}}{2n-1}\left[\frac{2n^4+3n^2-3n+1}{12n^5}\right].
\end{align}
Since it is easy to verify that all factors on the right side of (\ref{c2n2ineq}) are positive for $n\geq 2,$ it follows that $\nrel(C_{2n+1};2n^2)$ is positive for $n \geq 2$.

Thus for $n \geq 2$, $\nrel(C_{2n+1};2n^2-1)<0$ and $\nrel(C_{2n+1};2n^2)>0.$  We conclude by the Intermediate Value Theorem that $\nrel(C_{2n+1};p)$ has a root in the interval $\left(2n^2-1,2n^2\right).$
\end{proof}

\section{The closure of the collection of node reliability roots}
 
It was shown in \cite{BCConjecture} that all-terminal reliability roots are dense in the disk $|z-1|\leq 1$, but it is unknown whether they are dense in the entire complex plane, and it is suspected that they are not.  Hence the following result on the closure of the collection of node reliability roots is a surprising one.

\begin{theorem}\label{NodeRootsDense}
The collection of all node reliability roots is dense in the entire complex plane.
\end{theorem}

The remainder of this section is devoted to proving Theorem \ref{NodeRootsDense}.  The proof involves several elements, so we give a brief summary of our method here.  Once again we will make use of the connected set polynomial.  We will prove that the collection of connected set roots is dense in the complex plane, from which Theorem \ref{NodeRootsDense} follows easily.  We begin with a proof that the collection of connected set roots is dense in the complex plane if for every positive real number $r>0$ there is a connected set root whose distance from $-1$ is arbitrarily close to $r.$  In order to prove this result, we use the \textit{lexicographic product} with a complete graph to ``fan out'' a given connected set root evenly around the point $z=-1.$  We then present some background on the Beraha-Kahane-Weiss Theorem for finding limits of roots of recursive families of polynomials, which allows us to find a limiting curve of connected set roots that extends from the point $z=-1$ to infinity.  This gives us a limit of connected set roots at any distance from $-1,$ which completes the proof that connected set roots are dense in the entire complex plane.  The proof of Theorem \ref{NodeRootsDense} is straightforward from this point.  We begin with the definition of the lexicographic product of graphs.

\begin{definition}[see \cite{ProductHandbook}, for example]
The \textit{lexicographic product} or \textit{graph composition} $G\lexi H$ of graphs $G$ and $H$ is a graph on node set $V(G)\times V(H)$ such that nodes $(u,x)$ and $(v,y)$ are adjacent if and only if either
\begin{itemize}
\item $u$ is adjacent to $v$ in $G,$ or
\item $u=v$ and $x$ is adjacent to $y$ in $H.$
\end{itemize}
Intuitively, the lexicographic product $G\lexi H$ is the operation of replacing every node of $G$ with a copy of $H.$
\end{definition}

The next lemma gives a formula for the connected set polynomial of the lexicographic product $G\lexi H$ in terms of the connected set polynomials of the factors $G$ and $H.$  The special case of this formula for $H=K_n$ is all that we really need to prove our result on the closure of the collection of connected set roots.

\begin{lemma}\label{Lexicographic}
Let $G$ be a graph of order $n_G$ and let $H$ be a graph of order $n_H.$  The connected set polynomial of the lexicographic product $G\lexi H$ is given by
\[
\con(G\lexi H;x)=\con\left(G;(x+1)^{n_H}-1\right)+n_G\left[\con(H;x)-(x+1)^{n_H}+1\right].
\]
In particular, 
\[
\con(G\lexi K_n;x)=\con\left(G;(x+1)^n-1\right).
\]
\end{lemma}

\begin{proof}
Let $C_\star$ be a subset of $V(G\lexi H).$  Define
\[
C_G=\{v\in V(G) \colon \ (v,x)\in C_\star \mbox{ for some } x\in V(H)\}
\]
and for each $v\in C_G$ define
\[
C_{v}=\{x \colon \ (v,x)\in C_\star\}.
\]
We see that $C_\star$ is a connected set in $G\lexi H$ if and only if either
\begin{enumerate}[(i)]
\item $C_G$ is a connected set of $G$ of order at least two (in this case, for each $v\in C_G,$ the set $C_v$ can be any nonempty subset of nodes of $H$); or
\item $C_G=\{v\}$ for some $v$ and $C_v$ is a connected set of $H.$
\end{enumerate}

The connected sets of $G$ of order at least two are enumerated by 
\[
\con(G;x)-n_Gx,
\]
and hence the connected sets of $G\lexi H$ corresponding to case (i) are enumerated by
\[
\con(G;(x+1)^{n_H}-1)-n_G\left[(x+1)^{n_H}-1\right].
\] 
Meanwhile, the connected sets of $G\lexi H$ corresponding to case (ii) are enumerated by
\[
n_G\con(H;x).
\]
We conclude that
\begin{align*}
\con(G\lexi H;x)&=\con\left(G;(x+1)^{n_H}-1\right)+n_G\left[\con(H;x)-(x+1)^{n_H}+1\right]. \qedhere
\end{align*}
\end{proof}

The next lemma demonstrates that finding a family of graphs whose connected set roots extend from $-1$ to infinity is sufficient to show that the collection of all connected set roots is dense in the entire complex plane.  The proof is based on that of a similar result in \cite{Domination}, and uses lexicographic products with complete graphs to ``fan out'' existing connected set roots.

\begin{lemma}[adapted from \cite{Domination}, Theorem 11]\label{Density}
Suppose that for any $r>0$ and $\varepsilon>0$ there is some connected set root $z$ satisfying
\[
||z+1|-r|<\varepsilon;
\]
that is, for any positive real number $r$ we can find some connected set root $z$ whose distance from $-1$ is arbitrarily close to $r.$  Then the collection of all connected set roots is dense in the complex plane.
\end{lemma}

\begin{proof}
Let $r>0$ and $\theta\in[0,2\pi).$  It suffices to show that for any $\varepsilon>0$ there is a root $x$ of a connected set polynomial such that $x+1$ has modulus within $\varepsilon$ of $r$ and argument within $\varepsilon$ of $\theta.$  We may assume that $\varepsilon<r,$ so that $r-\varepsilon>0.$  We can choose $m$ large enough so that $\tfrac{\pi}{m}<\varepsilon,$ and hence for any complex number $w\neq 0,$ there is an $m^\mathrm{th}$ root of $w$ whose argument is within $\varepsilon$ of $\theta.$

By the supposition in the lemma statement, there is a connected set root $z$ of some graph $G$ that satisfies
\[
(r-\varepsilon)^m<|z+1|<(r+\varepsilon)^m.
\]
Consider the graph $G\lexi K_m.$   By Lemma \ref{Lexicographic}, any complex number $x$ satisfying 
\[
(x+1)^m-1=z
\]
is a root of $C(G\lexi K_m;x).$  For any such $x,$ we have
\[
|x+1|^m=|z+1|,
\]
so that 
\[
r-\varepsilon<|x+1|<r+\varepsilon.
\]
Further, for at least one such $x,$ the argument of $x+1$ is within $\varepsilon$ of $\theta$ by our choice of $m.$
\end{proof}

In order to show that connected set roots are dense in the complex plane, by Lemma \ref{Density} it now suffices to show that there is a limiting curve of connected set roots that extends from the point $-1$ to infinity.  We will need a precise definition for a limit of roots.

\begin{definition}
If $\{f_n(x)\colon\ n\in\mathbb{N}\}$ is a family of (complex) polynomials, we say that a number $z\in\mathbb{C}$ is a \textit{limit of roots} of $\{f_n(x)\colon\ n\in\mathbb{N}\}$  if there is a sequence $\{z_n \colon \ n\in\mathbb{N}\}$ such that $f_n(z_n)=0$ and $z_n\rightarrow z$ as $n\rightarrow\infty$.
\end{definition}

Under certain nondegeneracy conditions given in \cite{Beraha}, $z$ is a limit of roots of $\{f_n(z)\colon n\in\mathbb{N}\}$ if and only if either $f_n(z)=0$ for all sufficiently large $n$, or $z$ is a limit point of the set of all roots of the family.  The main result in \cite{Beraha} concerns limits of roots of certain recursively defined families of polynomials.  The solution of the recursion
\[
P_{n+k}(z)=-\sum_{i=1}^k f_i(z) P_{n+k-i}(z)
\]
depends on the roots of the characteristic equation
\[
Q_z(\lambda)=\lambda^{k}+\sum_{i=1}^k f_i(z) \lambda^{k-i}=0.
\]
Let these roots be $\lambda_1(z), \lambda_2(z),\hdots,\lambda_k(z),$ with possible repetitions.  If the $\lambda_i(z)$ are distinct for a particular $z,$ then
\begin{align}\label{RecursionSolution}
P_n(z)=\sum_{i=1}^k\alpha_i(z)\lambda_i(z)^n,
\end{align}
where $\alpha_i(z)$ are fixed polynomials determined by solving the system of equations that arises from letting $n=0,\hdots,k-1$ in (\ref{RecursionSolution}).  If there are repeated root values at $z,$ the solution is modified in the usual way (see \cite{Beraha}).  For example, if $\lambda_i(z)=\lambda_j(z),$ the term $\alpha_i\lambda_i^n+\alpha_j\lambda_j^n$ is replaced by $\alpha_{i_1}\lambda_i^n+n\alpha_{i_2}\lambda_i^{n-1}.$

Beraha, Kahane, and Weiss characterized the limits of roots of such a recursive family in \cite{Beraha}, and Brown and Hickman made the observation that any family of polynomials of the form (\ref{RecursionSolution}) satisfies such a recursion \cite{BrownLimit}.  This gives the following important theorem, which we refer to as the \textit{Beraha-Kahane-Weiss Theorem}.

\begin{theorem}[Beraha-Kahane-Weiss Theorem, cf. \cite{BrownLimit}]\label{Beraha}
Let
\begin{align*}
f_n(x)=\alpha_1(x)\lambda_1(x)^n+\alpha_2(x)\lambda_2(x)^n+\hdots+\alpha_k(x)\lambda_k(x)^n,
\end{align*}
where the $\alpha_i(x)$ and $\lambda_i(x)$ are fixed non-zero polynomials such that for no pair $i\neq j$ is $\lambda_i(x)=\omega\lambda_j(x)$ for some $\omega\in\mathbb{C}$ of unit modulus.  Then $z\in\mathbb{C}$ is a limit of roots of the family $\{f_n(x)\colon\ n\in\mathbb{N}\}$ if and only if either
\begin{enumerate}[\normalfont(i)]
\item two or more of the $\lambda_i(z)$ are of equal modulus, and strictly greater (in modulus) than the others; or
\item for some $j$, $\lambda_j(z)$ has modulus strictly greater than all the other $\lambda_i(z)$ have, and $\alpha_j(z)=0$.
\end{enumerate}
The same characterization holds when the characteristic equation of the associated recursion has repeated roots.  In particular, if the term $\alpha_i\lambda_i^n+\alpha_j\lambda_j^n$ in $f_n(x)$ is replaced by $\alpha_{i_1}\lambda_i^n+n\alpha_{i_2}\lambda_i^{n-1},$ the same conclusion holds.  In this case part {\normalfont{(ii)}} needs to be reworded slightly to: for some $j,$ $\lambda_j(z)$ has modulus strictly greater than all the other $\lambda_i(z)$ have, and $\alpha_{j_k}(z)=0$ for some $k.$
\end{theorem}

We now find a closed formula for a family of connected set polynomials on which we will apply the Beraha-Kahane-Weiss Theorem.  We start by computing the connected set polynomial of the path $P_n$ for $n\in\mathbb{N}.$  A connected set of order $k$ in $P_n$ must consist of $k$ consecutive nodes of the path, and there are $n-k+1$ sets of $k$ consecutive nodes in $P_n.$  Hence $c_k(P_n)=n-k+1$ for all $k\in\{1,\hdots,n\},$ and 
\[
\con(P_n;x)=\sum_{k=1}^n(n-k+1)x^k
\]  
A closed formula for $\con(P_n;x)$ is easily obtained.  We have
\begin{align*}
(x-1)\con(P_n;x)&=x\con(P_n;x)-\con(P_n;x)\\
&=x^{n+1}+x^n+\hdots +x^2-nx\\
&=\frac{x^{n+2}-x^2}{x-1}-nx,
\end{align*}
so that 
\[
\con(P_n;x)=\frac{x^{n+2}-x^2}{(x-1)^2}-\frac{nx}{x-1}.
\]

\begin{figure}
\centering{
\begin{overpic}[scale=0.5]{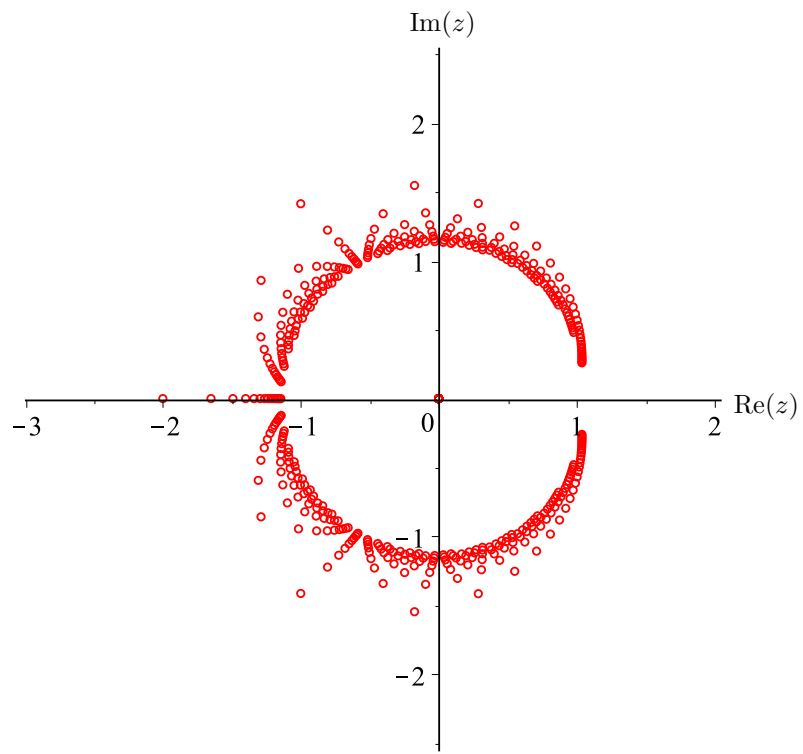}
\put(100,48){\footnotesize $\Re(z)$}
\put(56,100){\footnotesize $\Im(z)$}
\end{overpic}}
\caption{The connected set roots of $P_n$ for $n\in\{1,\hdots,30\}.$}
\label{PathRoots}
\end{figure}

The connected set roots of the path $P_n$ for $n\leq 30$ are shown in Figure \ref{PathRoots}.  It can be proven that the modulus of any connected set root of $P_n$ (for any $n$) is at most $2$ (see \cite{MolThesis}).  However, we see below that joining a single node to the path $P_n$ has a drastic effect on the connected set roots.  Given a graph $G$ and a node $v\not\in V(G),$ the \textit{join} $G+v$ is the graph on node set $V(G)\cup \{v\}$ and edge set $E(G)\cup \{\{u,v\}\colon\ u\in V(G)\}.$  That is, the node $v$ is added to $G$ along with an edge between $v$ and every node of $G.$

\begin{lemma}\label{JoinLemma}
Let $G$ be a graph on $n$ nodes.  The connected set polynomial of the join $G+v$ is given by
\[
\con(G+v;x)=\con(G;x)+x(x+1)^n.
\]
\end{lemma}

\begin{proof}
The connected sets of $G+v$ can be partitioned into those containing $v$ and those not containing $v.$  The connected sets of $G$ not containing $v$ are enumerated by $\con(G;x).$  The connected sets of $G+v$ containing $v$ correspond simply to the subsets of $V(G),$ since $v$ is adjacent to every node of $G.$  Explicitly, $U\subseteq V(G)$ corresponds to the connected set $U\cup \{v\}$ of $G+v.$  Hence the connected sets of $G+v$ containing $v$ are enumerated by $x(x+1)^n.$  Therefore,
\[
\con(G+v;x)=\con(G;x)+x(x+1)^n. \qedhere
\]
\end{proof}

The connected set roots of $P_n+v$ for $n\leq 50$ are shown in Figure \ref{PathJoinvRoots}, and one can see that there are roots that appear to grow large in several directions.  We are now ready to prove that the collection of connected set roots is dense in the complex plane.  The proof entails showing that the limiting curve of the connected set roots of the family of graphs $\{P_n+v\colon\ n\in\mathbb{N}\}$ extends from $-1$ to infinity.  The result then follows immediately by Lemma \ref{Density}.

\begin{figure}
\centering{
\begin{overpic}[scale=0.6]{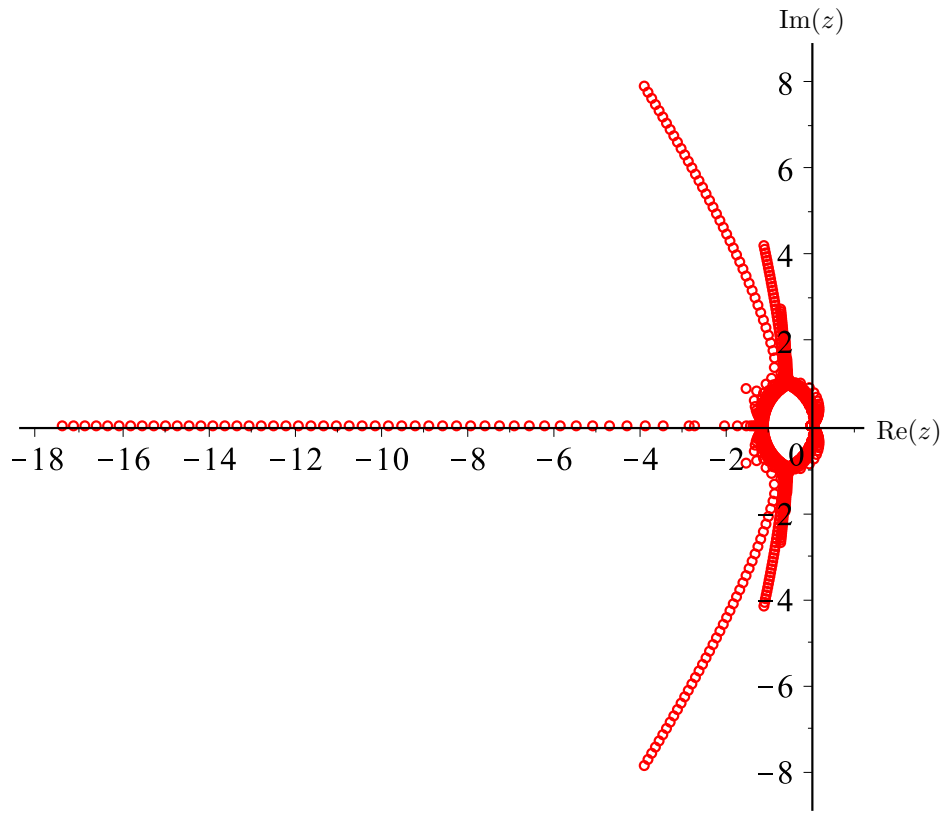}
\put(100,48.5){\footnotesize $\Re(z)$}
\put(89,95){\footnotesize $\Im(z)$}
\end{overpic}}
\caption{The connected set roots of $P_n+v$ for $n\in\{1,\hdots,50\}.$}
\label{PathJoinvRoots}
\end{figure}

\begin{theorem}\label{FinalDensity}
The collection of all connected set roots is dense in the entire complex plane, even if we restrict to connected graphs.
\end{theorem}

\begin{proof}
It suffices to show that the supposition of Lemma \ref{Density} is true.  We do so by proving that the limits of roots of the family $\{\con(P_n+v;x)\colon\ n\geq 1\}$ include the points on the line $\Re(z)=-\tfrac{1}{2}$ of modulus at least one, the points on the circle $|z+1|=1$ with $\Re(z)\geq -\tfrac{1}{2},$ and the points on the circle $|z|=1$ with $\Re(z)\leq -\tfrac{1}{2}.$  See Figure \ref{LimitCurve} for an illustration of the limiting curve, which clearly extends from $-1$ to $\infty.$  For $n\geq 1,$ the connected set polynomial of $P_n+v$ is given by
\begin{align*}
\con(P_n+v;x)&=\con(P_n;x)+x(x+1)^n\\
&=\frac{x^{n+2}-x^2}{(x-1)^2}-\frac{nx}{x-1}+x(x+1)^n.
\end{align*}
Consider the polynomial $f_n(x)=(x-1)^2\con(P_n+v;x).$  We multiply by $(x-1)^2$ to clear the denominators of the rational terms and this adds only a simple root at $x=1.$  We rewrite $f_n(x)$ as follows:
\begin{align*}
f_n(x)&=x^{n+2}-x^2-nx(x-1)+(x-1)^2x(x+1)^n\\
&=x(x-1)^2(x+1)^n+x^2\cdot x^n-x^2-nx(x-1)\\
&=\alpha_1\lambda_1^n+\alpha_{2}\lambda_2^n+\alpha_{3_1}\lambda_3^n+n\alpha_{3_2}\lambda_3^{n-1},
\end{align*}
where
\[
\alpha_1=x(x-1)^2,\ \alpha_{2}=x^2,\ \alpha_{3_1}=-x^2, \mbox{ and } \alpha_{3_2}=-x(x-1);
\]
and
\[
\lambda_1=x+1,\ \lambda_2=x, \mbox{ and } \lambda_3=1.
\]
Clearly no $\alpha_i$ is identically zero and no $\lambda_i=\omega\lambda_j$ for $i\neq j$ and some complex number $\omega$ of unit modulus, so the nondegeneracy conditions of Theorem \ref{Beraha} are satisfied.  Applying part (i) of Theorem \ref{Beraha} involves three cases.  Figure \ref{LimitCurve} is provided to aid the reader in seeing the characterizations below.

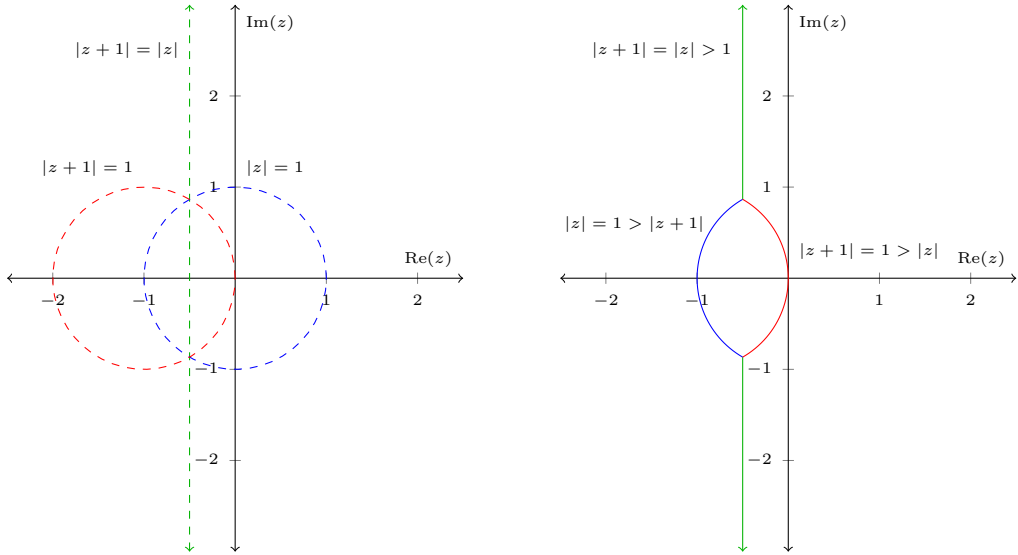
\begin{figure}
\begin{center}
\begin{tikzpicture}
\begin{axis}[
    scale only axis,
    axis equal image,
    axis lines=middle,
    x axis line style={<->},
    y axis line style={<->},
    xlabel={\tiny $\Re(z)$},
    ylabel={\tiny $\Im(z)$},
    ytick={-2,-1,1,2},
    yticklabels={\tiny $-2$,\tiny $-1$,\tiny $1$,\tiny $2$},
    xtick={-2,-1,1,2},
    xticklabels={\tiny $-2$,\tiny $-1$,\tiny $1$,\tiny $2$},
    ymin=-3,
    ymax=3,
    xmin=-2.5,
    xmax=2.5,
    samples=50
]
     \draw[dashed,green!70!black,style=<->] (axis cs:-0.5,-3)--(axis cs:-0.5,3);
     \addplot [dashed,domain=0:2*pi,samples=50,blue]({cos(deg(x))},{sin(deg(x))});
     \addplot [dashed,domain=0:2*pi,samples=50,red]({cos(deg(x))-1},{sin(deg(x))});
     \node[left] at (axis cs:-0.5,2.5) {\tiny $|z+1|=|z|$};
     \node[above left] at (axis cs:-1,1) {\tiny $|z+1|=1$};
     \node[above right] at (axis cs:0,1) {\tiny $|z|=1$};
\end{axis}
\end{tikzpicture}
\hspace{1cm}
\begin{tikzpicture}
\begin{axis}[
    scale only axis,
    axis equal image,
    axis lines=middle,
    x axis line style={<->},
    y axis line style={<->},
    xlabel={\tiny $\Re(z)$},
    ylabel={\tiny $\Im(z)$},
    ytick={-2,-1,1,2},
    yticklabels={\tiny $-2$,\tiny $-1$,\tiny $1$,\tiny $2$},
    xtick={-2,-1,1,2},
    xticklabels={\tiny $-2$, \tiny $-1$,\tiny $1$,\tiny $2$},
    ymin=-3,
    ymax=3,
    xmin=-2.5,
    xmax=2.5,
    samples=50
]
     \draw[green!70!black,style=->] (axis cs:-0.5,0.86602540378)--(axis cs:-0.5,3);
     \draw[green!70!black,style=<-] (axis cs:-0.5,-3)--(axis cs:-0.5,-0.86602540378);
     \addplot [domain=2*pi/3:4*pi/3,samples=50,blue]({cos(deg(x))},{sin(deg(x))});
     \addplot [domain=-pi/3:pi/3,samples=50,red]({cos(deg(x))-1},{sin(deg(x))});
     \node[left] at (axis cs:-0.5,2.5) {\tiny $|z+1|=|z|>1$};
     \node[left] at (axis cs:-0.8,0.6) {\tiny $|z|=1>|z+1|$};
     \node[right] at (axis cs:0,0.3) {\tiny $|z+1|=1>|z|$};
\end{axis}
\end{tikzpicture}
\end{center}
\caption[The limiting curve for the connected set roots of $P_n+v.$]{The curves $|z+1|=|z|,$ $|z+1|=1,$ and $|z|=1$ (left), and the limiting curve for the connected set roots of $P_n+v$ (right).}
\label{LimitCurve}
\end{figure}

\noindent
Case (i): $|\lambda_1|=|\lambda_2|\geq |\lambda_3|$

The condition $|z+1|=|z|$ is true if and only if $z$ is equidistant from $-1$ and $0$; that is, if and only if $\Re(z)=-\tfrac{1}{2}.$  Further, when $|z+1|=|z|$ we have $|z+1|\geq1$ and $|z|\geq 1$ if and only  if $z$ has modulus at least one.  Hence, we have $|z+1|=|z|\geq 1$ if and only if $z$ lies on the line $\Re(z)=-\tfrac{1}{2}$ and $z$ has modulus at least one.

\noindent
Case (ii): $|\lambda_1|=|\lambda_3|\geq |\lambda_2|$

The condition $|z+1|=1$ is true if and only if $z$ lies on the circle of radius $1$ centred at the point $-1.$  Further, when $|z+1|=1$ we have $|z+1|\geq |z|$ and $1\geq |z|$ if and only if $\Re(z)\geq -\tfrac{1}{2}.$  Hence, we have $|z+1|=1\geq |z|$ if and only if $z$ lies on the circle of radius $1$ centred at $-1$ and $\Re(z)\geq -\tfrac{1}{2}.$

\noindent
Case (iii): $|\lambda_2|=|\lambda_3|\geq |\lambda_1|$

The condition $|z|=1$ is true if and only if $z$ lies on the circle of radius $1$ centred at the point $0.$  Further, when $|z|=1$ we have $|z|\geq |z+1|$ and $1\geq |z+1|$ if and only if $\Re(z)\leq -\tfrac{1}{2}.$  Hence, we have $|z|=1\geq |z+1|$ if and only if $z$ lies on the circle of radius $1$ centred at $0$ and $\Re(z)\leq -\tfrac{1}{2}.$

Since the limiting curve of the connected set roots of the graphs $P_n+v$ extends continuously from the point $-1$ to infinity and we can find a connected set root arbitrarily close to any point on this curve, the supposition of Lemma \ref{Density} is satisfied.  We conclude that the collection of connected set roots is dense in the entire complex plane.  The result holds even if we restrict to connected graphs as we have only used the connected set roots of the connected graphs $(P_n+v)\lexi K_m$ for $n,m\in\mathbb{N}.$
\end{proof}

Finally, using Theorem \ref{FinalDensity} we can prove Theorem \ref{NodeRootsDense}, which states that the collection of all node reliability roots is dense in the complex plane.  This is potentially another significant difference between node reliability and all-terminal reliability.  While all-terminal reliability roots were shown to be dense in the disk $|z-1|\leq 1$ in \cite{BCConjecture}, the largest known distance of an all-terminal reliability root outside of this disk is $0.1134860896$ (see \cite{MolATRRoots}).  It seems as though all-terminal reliability roots are bounded in modulus by some constant, which would make them far from dense in the complex plane.  If this is indeed the case then we have found another striking difference between all-terminal reliability and node reliability.

\begin{proof}[Proof of Theorem \ref{NodeRootsDense}]
Let $z\in\mathbb{C}$ and let $\varepsilon>0.$  We will find a complex number $\tilde{z}$ and a graph $G$ such that $|z-\tilde{z}|<\varepsilon$ and $\nrel(G;\tilde{z})=0.$  Recall that 
\[
\con(G;x)=(1+x)^n\cdot\nrel\left(G;\tfrac{x}{1+x}\right),
\]
so that any connected set root $x\neq -1$ of the graph $G$ yields a node reliability root $\tfrac{x}{1+x}$ of $G.$  The complex function $f(x)=\tfrac{x}{1+x}$ is a M\"obius transformation and hence it is one-to-one and continuous on its domain.  Let $x=f^{-1}(z)=\frac{z}{1-z}$ (we may assume that $z\neq 1$).  Since $f$ is continuous, we can find $\delta>0$ such that $|x-\tilde{x}|<\delta$ implies $|f(x)-f(\tilde{x})|=|z-f(\tilde{x})|<\varepsilon.$  By Theorem \ref{FinalDensity}, the connected set roots of connected graphs are dense in the complex plane, and hence there is some connected graph $G$ with connected set root $\tilde{x}\neq -1$ satisfying $|x-\tilde{x}|<\delta.$  The complex number $f(\tilde{x})=\frac{\tilde{x}}{1+\tilde{x}}$ satisfies $|z-f(\tilde{x})|<\varepsilon$ and $\nrel(G;f(\tilde{x}))=0.$
\end{proof}
 
\section{Conclusion and open problems}

Note that while all-terminal reliability is identically $0$ for any disconnected graph, node reliability and the connected set polynomial are nonzero for disconnected graphs, and are in fact quite interesting.  The connected set polynomial of a disconnected graph is given simply by the sum of the connected set polynomials of its components, and the node reliability of a disconnected graph can also be written in terms of the node reliabilities of its components.  Let $G$ be a graph of order $n$ with connected components $G_1,\hdots,G_k$ of orders $n_1,\hdots,n_k,$ respectively.  The connected set polynomial of $G$ is given by
\[
\con(G;x)=\sum_{i=1}^k\con(G_i;x),
\]
and the node reliability of $G$ is given by
\[
\nrel(G;p)=\sum_{i=1}^k (1-p)^{n-n_i}\nrel(G_i;p).
\]

Lemma \ref{NonrealConnectedSetRoot} implies that the connected set polynomial of a \textit{connected} graph has all real roots if and only if it has degree at most $2,$ and the same result for node reliability follows readily in Theorem \ref{NonrealNodeRoot}.  However, we have found \textit{disconnected} graphs whose connected set polynomials have degree $3$ and still have all real roots.  Consider the graph $K_3\cup kK_2,$ i.e.\ the disjoint union of $K_3$ and $k$ copies of $K_2.$  We have
\begin{align*}
\con(K_3\cup kK_2;x)&=\con(K_3;x)+k\con(K_2;x)\\
&=x^3+3x^2+3x+k(x^2+2x)\\
&=x^3+(3+k)x^2+(3+2k)x\\
&=x(x^2+(3+k)x+(3+2k))
\end{align*}
The discriminant of the quadratic factor in the above expression factors to
\[
(k-3)(k+1).
\]
Therefore, $\con(K_3\cup kK_2;x)$ has all real roots for $k\geq 3.$  One can show similarly that $\con(P_3\cup kK_2;x)$ has all real roots for $k\geq 6.$  All node reliability roots of these graphs will be real as well.  Are there disconnected graphs whose connected set polynomials (and node reliabilities) have degree greater than $3$ and still have all real roots?  What about arbitrarily high degree?

We turn now to the closure of the collection of real node reliability roots. Disconnected graphs will allow us to show that the collection of real node reliability roots is dense in the interval $[1,2]$.  We remark that from \cite{BCConjecture} it is known that the closure of the collection of real roots of (nonzero) all-terminal reliability polynomials is exactly the set $\{0\} \cup [1,2]$.  

\begin{theorem}\label{RealDense}
The closure of the collection of real node reliability roots contains the interval $[1,2]$.
\end{theorem}  

\begin{proof}
Let $a,b\in\mathbb{N}$ with $b\leq a\leq 2b.$  It suffices to show that $\tfrac{a}{b}$ is a node reliability root of some graph.  Let
\[
G_{a,b}=(a-b)K_2\cup (2b-a)K_1,
\] 
that is, the disjoint union of $a-b$ copies of $K_2$ and $2b-a$ copies of $K_1.$  Since $G_{a,b}$ has $a-b$ edges and $2(a-b)+2b-a=a$ nodes, and no connected sets of order $3$ or more, its node reliability is given by
\begin{align*}
\nrel(G_{a,b};p)&=(a-b)p^2(1-p)^{a-2}+ap(1-p)^{a-1}\\
&=p(1-p)^{a-2}\left[(a-b)p+a(1-p)\right]\\
&=p(1-p)^{a-2}\left[a-bp\right].
\end{align*}
Therefore, the graph $G_{a,b}$ has node reliability root $\tfrac{a}{b},$ and we are done.
\end{proof}

The previous result begs the question: is $[1,2]$ in the closure of the collection of {\it real} node reliability roots of connected graphs?  What is the closure of the collection of real node reliability roots (for connected or disconnected graphs)?

Finally, we can also show that the node reliability roots of disconnected graphs are dense in the complex plane, which is surprising as asymptotically almost all graphs of order $n$ are connected.  The proof (details of which are omitted) is very similar to that of Theorem \ref{FinalDensity}, but we use the disconnected graphs $P_n\cup K_n$ in place of the connected graphs $P_n+v.$ 
 
\section*{Acknowledgements}
 
Research of Jason I.\ Brown is partially supported by grant RGPIN 170450-2013 from Natural Sciences and Engineering Research Council of Canada (NSERC).  Research of Lucas Mol is partially supported by an Alexander Graham Bell Canada Graduate Scholarship from NSERC.

\singlespacing

\bibliographystyle{plain}
\bibliography{2016ThesisBib}

\end{document}